\documentclass{amsart}
 \newtheorem{thm}{Theorem}[section]
 \newtheorem{cor}[thm]{Corollary}
 \newtheorem{lem}[thm]{Lemma}
 \newtheorem{prop}[thm]{Proposition}
 \theoremstyle{definition}
 \newtheorem{defn}[thm]{Definition}
 \theoremstyle{remark}
 \newtheorem{rem}[thm]{Remark}
 
 \numberwithin{equation}{section}

\begin{document}

%-------------------------------------------------------------------------
% editorial commands: to be inserted by the editorial office
%
%\firstpage{1} \volume{228} \Copyrightyear{2004} \DOI{003-0001}
%\seriesextraline{This is the Concrete Title of this Book\br H.E. R and S.T.C. W, Eds.}
%\firstpage{1}

\title[COMPACTNESS IN FUNCTION SPACES]
 {COMPACTNESS IN FUNCTION SPACES}

%----------Author 1
\author[\v Lubica Hol\'a]{\v Lubica Hol\'a}

\address{%
Academy of Sciences,\\
Institute of Mathematics\\
\v Stef\'anikova 49,\\
81473 Bratislava,\\
Slovakia}

\email{hola@mat.savba.sk}

\thanks{This work was completed with the support of our
\TeX-pert.}
%----------Author 2
\author{Du\v san Hol\'y}
\address{Department of Mathematics and Computer Science, Faculty of Education,
Trnava University,
Priemyseln\'a 4, 918 43 Trnava, Slovakia}
\email{dusan.holy@truni.sk}


\date{January 1, 2004}
\dedicatory{}
%%% ----------------------------------------------------------------------

\begin{abstract}
Let $X$ be a locally compact topological space, $(Y,d)$ be a boundedly compact metric space and $LB(X,Y)$ be the space of all locally bounded functions from $X$ to $Y$. We characterize compact sets  in $LB(X,Y)$ equipped with the topology of uniform convergence on compacta. From our results we obtain the following interesting facts for $X$ compact:

$\bullet$ If $(f_n)_n$ is a sequence of uniformly bounded finitely equicontinuous functions of Baire class $\alpha$ from $X$ to $\Bbb R$, then
there is a uniformly convergent subsequence $(f_{n_k})_k$;

$\bullet$ If $(f_n)_n$ is a sequence of uniformly bounded finitely equicontinuous lower (upper) semicontinuous functions from $X$ to $\Bbb R$, then
there is a uniformly convergent subsequence $(f_{n_k})_k$;

$\bullet$ If $(f_n)_n$ is a sequence of uniformly bounded finitely equicontinuous quasicontinuous functions from $X$ to $Y$, then there is a uniformly convergent subsequence $(f_{n_k})_k$.
\end{abstract}

%%% ----------------------------------------------------------------------
\maketitle
%%% ----------------------------------------------------------------------
%\tableofcontents
\section{Introduction}

The Arzel\` a-Ascoli theorem is a well known and fundamental result of mathematical analysis concerning compactness  of a family of real-valued continuous functions defined on a closed and bounded interval equipped with the topology of uniform convergence.
It claims: If a sequence $(f_n)_n$ of continuous real-valued functions defined on $[0, 1]$ is uniformly bounded
and equicontinuous then there is a subsequence $(f_{n_k})_k$  which is uniformly convergent. The converse is also true.
The notion of equicontinuity was introduced  by the Italian mathematicians Cesare Arzel\` a \cite{Ar1,Ar2} and Giulio Ascoli \cite{As}.
\bigskip

In our paper we introduce the notion of finite equicontinuity as a key notion to characterize compact subsets of the space of locally bounded functions equipped with the topology of uniform convergence on compacta.  In the class of continuous functions the notions of equicontinuity and finite equicontinuity coincide.

We apply our results to the space of functions of Baire class $\alpha$, quasicontinuous, cliquish, upper semicontinuous, lower semicontinuous, etc.

In \cite{Hol} Ascoli-type theorem for quasicontinuous locally bounded functions was studied and in  \cite{HH4} we proved Ascoli-type theorems for quasicontinuous subcontinuous functions. In \cite{HH6} we studied compact subsets of quasicontinuous functions and in \cite{HH5} compact subsets of minimal usco and minimal cusco maps equipped with the topology of uniform convergence on compact sets. Notice that some relative results using quasicontinuous and subcontinuous selections can be found in \cite{Hol}. Ascoli-type theorems for so-called densely continuous forms and locally bounded densely continuous forms were proved in \cite{HM} and \cite{Mc}.

\bigskip
Here is the main result of our paper:

Let $X$ be a locally compact space, $(Y,d)$ be a boundedly compact metric space,  $LB(X,Y)$ be the space of locally bounded functions from $X$ to $Y$  and $\tau_{UC}$ be the topology of uniform convergence on compact sets. A subfamily $\mathcal E$ of the space  $LB(X,Y)$  is compact
in $(LB(X,Y),\tau_{UC})$  if and only if $\mathcal E$ is closed, pointwise bounded and finitely equicontinuous.

\bigskip

\section{Characterization of compactness in function spaces}

\bigskip

In what follows let $\Bbb N$ be the set of natural numbers, $\Bbb R$ be the space of real numbers with the usual metric.

Let $X$ be a topological space, $(Y,d)$ be a metric space and $F(X,Y)$ be the space of all functions from $X$ to $Y$.

The open $d$-ball with center $z_0\in Y$ and radius $\varepsilon
>0$ will be denoted by $S(z_0,\varepsilon)$ and the
$\varepsilon $-parallel body $\bigcup_{a\in A}S(a,\varepsilon)$ for
a subset $A$ of $Y$ will be denoted by $S(A,\varepsilon)$. The closed $d$-ball with center $z_0\in Y$ and radius $\varepsilon
>0$ will be denoted by $B(z_0,\varepsilon)$.

\bigskip

We will define the topology $\tau_p$ of
pointwise convergence on $F(X,Y)$. The topology $\tau_p$ of pointwise
convergence on $F(X,Y)$ is induced by the uniformity $\frak U_p$ of
pointwise convergence which has a base consisting of sets of the
form

$$W(A,\varepsilon )=\{(f,g):\ \forall\ x\in A\ \ d(f(x),g(x))<
\varepsilon \},$$

where $A$ is a finite subset of $X$ and $\varepsilon >0$. The general
$\tau_p$-basic neighborhood of $f\in F(X,Y)$ will be denoted by
$W(f,A,\varepsilon )$, i.e. $W(f,A,\varepsilon
)=W(A,\varepsilon )[f]$.
$$W(f,A,\varepsilon )=\{g \in F(X,Y):\ d(f(x),g(x)) <
\varepsilon\ \text{for every}\ x \in A\}.$$

\bigskip

We will define the topology $\tau_{UC}$ of uniform convergence on
compact sets on $F(X,Y)$. This topology is induced by the
uniformity $\frak U_{UC}$  which has a base consisting of sets of the
form

$$W(K,\varepsilon )=\{(f,g):\ \forall\ x\in K\ \ d(f(x),g(x))<
\varepsilon \},$$

where $K$ is a compact subset of $X$ and $\varepsilon >0$. The general
$\tau_{UC}$-basic neighborhood of $f\in F(X,Y)$ will be denoted by
$W(f,K,\varepsilon )$, i.e. $W(f,K,\varepsilon
)=W(K,\varepsilon )[f]$.
$$W(f,K,\varepsilon) = \{g \in F(X,Y):\ d(f(x),g(x)) < \varepsilon\ \text{for every}\ x \in K\}.$$

\bigskip

Finally we will define the topology $\tau_U$ of uniform
convergence on $F(X,Y)$.  Let $\varrho$ be the (extended-valued)
metric on $F(X,Y)$ defined by
$$\varrho(f,g)=\sup\{d(f(x),g(x)):\ x\in X\},$$ for each
$f,g\in F(X,Y)$. Then the topology of uniform convergence for
the space $F(X,Y)$ is the topology generated by the metric $\varrho$.

\bigskip

\begin{defn}\label{AAA1}
	Let $X$ be a topological space and $(Y,d)$ be a metric space. We say that a subset $\mathcal{E}$  of $F(X,Y)$ is finitely equicontinuous at a point $x$ of $X$ provided that for every $\varepsilon >0$, there exists a finite family $\mathcal B$ of subsets of $X$ such that $\cup\mathcal B$ is a neighbourhood of $x$ and such that for every $f\in\mathcal{E}$, for every $B\in\mathcal B$ and for every $p,q\in B$, $d(f(p),f(q))<\varepsilon$. Then $\mathcal{E}$ is finitely equicontinuous provided that it is finitely equicontinuous at every point of $X$.
\end{defn}

\bigskip

If $f$ is a function from a topological space $X$ to a metric space $(Y,d)$, we say that $f$ is locally bounded, if for every $x\in X$ there is an open set $U\subset X$, $x \in U$ such that $f(U)= \{f(u):\ u \in U\}$ is a bounded subset of $(Y,d)$.

Denote by $LB(X,Y)$ the space of all locally bounded functions from a topological space $X$ to a metric space $(Y,d)$.

\begin{rem}\label{AAA2}
	It is easy to verify that $LB(X,Y)$ is a closed subset of $(F(X,Y),\tau_U)$ and if $X$ is locally compact, then $LB(X,Y)$ is also a closed set in $(F(X,Y),\tau_{UC})$.
\end{rem}

\begin{rem}\label{AAA3}
	Notice that if $X$ is a topological space,  $(Y,d)$ is  a metric space and  $\mathcal{E} \subset F(X,Y)$ is finitely equicontinuous, then $\mathcal{E} \subset LB(X,Y)$.
\end{rem}

\begin{prop}\label{AAA4}
	Let $X$ be a topological space and $(Y,d)$ be a metric space. If $\mathcal{E}$ is  finitely equicontinuous   subset of $C(X,Y)$, then $\mathcal{E}$ is equicontinuous.
\end{prop}
\begin{proof}
	Let $x \in X$ and let $\varepsilon > 0$. 	
	Since $\mathcal{E}$ is finitely equicontinuous, there exists a finite family $\mathcal B$ of nonempty subsets  of $X$ such that $\cup\mathcal B$ is a neighborhood of $x$ and such that for every $f\in \mathcal{E}$, for every $B\in\mathcal B$ and for every $p,q\in B$, $d(f(p),f(q))<\frac\epsilon 2$. Without loss of generality we can suppose that $x\in  \overline B$ for every $B\in\mathcal B$.
	We will show that for every $z\in\cup\mathcal B$ and for every $f\in\mathcal{E}, d(f(x),f(z)) < \epsilon.$
	
	Let $z\in\cup\mathcal B$ and let $f\in\mathcal{E}$. There is $B\in\mathcal B$ such that $z\in B$.  There is a neighbourhood $U$ of $x$ such that $d(f(x),f(s)) < \epsilon/2$ for every $s \in U$.  Since $U \cap B \ne \emptyset$, $d(f(x),f(z)) <\epsilon$.
\end{proof}

\bigskip

\begin{lem}\label{AAA5}
	Let $X$ be a topological space and $(Y,d)$ be a metric space. Let $\mathcal{E}$ be a finitely equicontinuous subset of $F(X,Y)$. Then the closure of $\mathcal{E}$ relative to the topology $\tau_p$ of pointwise convergence is also finitely equicontinuous subset of $F(X,Y)$.
\end{lem}
\begin{proof}
	Let $x\in X$ and let $\varepsilon >0$. Since $\mathcal{E}$ is finitely equicontinuous at $x$, there exists a finite family $\mathcal B$ of subsets of $X$ such that $\cup\mathcal B$ is a neighbourhood of $x$ and such that for every $f\in\mathcal{E}$, for every $B\in\mathcal B$ and for every $p,q\in B$, $d(f(p),f(q))<\frac\varepsilon 3$.
	
	Let $f\in\overline{\mathcal{E}}$. There is a net  $\{f_\sigma:\ \sigma \in \Sigma\} \subset \mathcal{E}$ pointwise convergent to $f$. Let $B\in\mathcal B$ and  $p,q\in B$. There is $\sigma\in\Sigma$ such that $d(f(p),f_\sigma (p))<\frac\varepsilon 3$ and $d(f(q),f_\sigma (q))<\frac\varepsilon 3$. Then
	\begin{align*}
	&d(f(p),f(q))\leq\\ \leq
	&d(f(p),f_\sigma(p))+d(f_\sigma(p)),f_\sigma(q))+
	d(f_\sigma(q)),f(q))<\\ <&\frac\varepsilon 3+\frac\varepsilon 3+\frac\varepsilon 3 =\varepsilon .
	\end{align*}
\end{proof}

\bigskip
\begin{prop}\label{AAA6}
	Let $X$ be a topological space and $(Y,d)$ be a metric space such that every bounded set is totally bounded. If $\mathcal{E}$ is a finitely equicontinuous  subset of $F(X,Y)$ and $f$ is locally bounded, then $\mathcal{E}\cup \{f\}$ is finitely equicontinuous.
\end{prop}

\begin{proof}
	Let $x\in X$ and let $\varepsilon >0$. Since $\mathcal{E}$ is finitely equicontinuous at $x$ there exists a finite family $\mathcal B=\{B_1,B_2,...,B_n\}$ of subsets of $X$ such that $\cup\mathcal B$ is a neighbourhood of $x$ and such that for every $g\in\mathcal{E}$, for every $B\in\mathcal B$ and for every $p,q\in B$, $d(g(p),g(q))<\varepsilon$. Let $U$ be a  neighbourhood of $x$ such that $U\subset\cup\mathcal B$ and $f(U)$ is bounded. Fix an element $y_0$ of $Y$ and let $r>0$ be such that the set $f(U)$ is a subset of the closed ball $B(y_0,r)$.
	Let $V_1,\ V_2,\ ...,\ V_m$ be a finite open cover of $B(y_0,r)$ where radius of members of this cover is less than $\varepsilon $. For every $j\in\{1,\ 2,\ ...,\ m\}$ put $H_j= f^{-1}(V_j)$. Then $\mathcal H=\{H_j :\ j\in\{1,\ 2,\ ...,\ m\}\}$
	is a finite family of subsets of $X$ where $U\subset\cup\mathcal H$.
	For every $i\in\{1,\ 2,\ ...,\ n\},\ j\in\{1,\ 2,\ ...,\ m\}$ put
	$D_{i,j}=B_i\cap H_j$.
	
	Denote by $\mathcal{D}$ the family containing all nonempty sets $D_{i,j}$ where $i\in\{1,\ 2,\ ...,\ n\},\ j\in\{1,\ 2,\ ...,\ m\}$. Since $U\subset\cup\mathcal B$ and $U\subset\cup\mathcal H$, we have that $\cup\mathcal D$ is a neighbourhood of $x$. Evidently for every $g\in\mathcal{E}\cup \{f\}$, for every $D\in\mathcal D$ and for every $p,q\in D$, $d(g(p),g(q))<\epsilon$.
\end{proof}

\bigskip
\bigskip

\begin{lem}\label{AAA7}
	Let $X$ be a topological space and $(Y,d)$ be a metric space. Let $\{f_\sigma  :\ \sigma \in \Sigma\}$ be a net in $F(X,Y)$ pointwise convergent to $f\in F(X,Y)$ and let the set $\{f_\sigma :\ \sigma \in \Sigma\}$ be finitely equicontinuous. Then $\{f_\sigma :\sigma \in\Sigma\}$ converges to $f$ also in $(F(X,Y),\tau_{UC})$.
\end{lem}

\begin{proof}
	Let $\{f_\sigma :\ \sigma\in \Sigma\}$ be a net in $F(X,Y)$ which $\tau_p$-converges to $f\in F(X,Y)$. We claim that $\{f_\sigma:\ \sigma \in \Sigma\}$ converges to $f$ also in $(F(X,Y),\tau_{UC})$.
	If not there exists a compact set $A$ in $X$ and $\epsilon > 0$ such that $\{f_\sigma:\ \sigma \in \Sigma\}$ is not residually in $W(A,2\epsilon)[f]$.
	
	For every $\sigma \in \Sigma$ there is $\beta_\sigma \ge \sigma$ and $x_\sigma \in A$ such that $d(f(x_\sigma),f_{\beta_\sigma}(x_\sigma)) > \epsilon$. Let $x$ be a cluster point of $\{x_\sigma:\ \sigma \in \Sigma\}$.
	
	By Lemma \ref{AAA5} the family $\mathcal E = \{f_\sigma :\ \sigma \in \Sigma\} \cup \{f\}$ is finitely equicontinuous. There is a finite family $\mathcal B$ of nonempty subsets of $X$  such that $\cup \mathcal B$ is a neighbourhood of $a$ and for every $h \in \mathcal E$, for every $B \in \mathcal B$ and for every $p, q \in B$, $d(h(p),h(q)) < \epsilon/3$.
	
	Without loss of generality we can suppose that $x_\sigma \in \cup \mathcal B$ for every $\sigma \in \Sigma$ and also since $\mathcal B$ is finite, we can suppose that there is $B \in \mathcal B$ such that $x_\sigma \in B$ for every $\sigma \in \Sigma$.
	
	Choose a $x_\eta$. There exist $\sigma \in \Sigma$ such that $d(f(x_\eta),f_\delta(x_\eta)) < \epsilon/3$ for every $\delta \ge \sigma$. Then
	\begin{align*}
	&d(f(x_\sigma),f_{\beta_\sigma}(x_\sigma)) \le\\ \le &d(f(x_\sigma),f(x_\eta)) + d(f(x_\eta),f_{\beta_\sigma}(x_\eta)) + d(f_{\beta_\sigma}(x_\eta),f_{\beta_\sigma}(x_\sigma))<\\ <
	&\epsilon,
	\end{align*}
	a contradiction.
	
	Thus $\{f_\sigma:\ \sigma \in \Sigma\}$ converges to $f$ in $(F(X,Y),\tau_{UC})$.
\end{proof}

\bigskip

Let $\mathcal{E} \subset F(X,Y)$ and let $x\in X$, denote by $\mathcal E[x]$ the set $\{f(x)\in Y:\
f\in \mathcal E\}$.

We say that a subset $\mathcal E$ of $F(X,Y)$ is pointwise bounded provided for every $x \in X$, $\mathcal E[x]$ is bounded in $(Y,d)$.

\bigskip

We say that a metric space $(Y,d)$ is boundedly compact \cite{Be} if every closed bounded subset is compact. Therefore $(Y,d)$ is a locally compact, separable metric space and $d$ is complete. In fact, any locally compact, separable metric space has a compatible metric $d$ such that $(Y,d)$ is a boundedly compact space \cite{Va}.

\bigskip

\begin{thm}\label{AAA8}
	Let $X$ be a locally compact space and $(Y,d)$ be a  boundedly compact metric space. A subset $\mathcal{E} \subset (LB(X,Y),\tau_{UC})$  is compact if and only if it is closed, pointwise bounded and finitely equicontinuous.
\end{thm}
\begin{proof}
	Let $\mathcal{E}$ be a closed, pointwise bounded and finitely equicontinuous subset of $LB(X,Y)$. Since $\mathcal{E}$ is pointwise bounded  and the metric space $(Y,d)$ is  boundedly compact, for every $x\in X$ there exists a compact set $B_x$ such that $\mathcal{E}[x]\subset B_x$. The product $H = \prod_{x\in X}B_x$ is a compact subset of $Y^X =\prod_{x\in X}\{Y_x:\ Y_x = Y\ \text{for every}\  x \in X\}$ with the relative product topology.
	
	Let $\{f_\sigma:\ \sigma \in \Sigma\}$ be a net in $\mathcal E$. Now we will show that there is a function from $ \mathcal E$ which is a cluster point of $\{f_\sigma:\ \sigma \in \Sigma\}$ in $(LB(X,Y),\tau_{UC})$.
	
	There is a function $f \in H$ such that $f$ is a cluster point of $\{f_\sigma:\ \sigma \in \Sigma\}$ in $H$. Without loss of generality we can suppose that $\{f_\sigma:\ \sigma \in \Sigma\}$ converges to $f$ in $H$. By Lemma \ref{AAA7} $\{f_\sigma:\ \sigma \in \Sigma\}$ converges to $f$ in $(F(X,Y),\tau_{UC})$. Since $\mathcal E \subset LB(X,Y)$, by Remark \ref{AAA2} also $f \in LB(X,Y)$ and thus $f \in \mathcal E$, since $\mathcal E$ is closed in $(LB(X,Y),\tau_{UC})$.
	
	For the converse, suppose that $\mathcal{E}$ is a compact subset  of $(LB(X,Y),\tau_{UC})$. The set $\mathcal{E}$ is closed because $(LB(X,Y),\tau_{UC} )$ is a Hausdorff space.
	
	For every $x \in X$ the  evaluation map at $x$,  $e_x: LB(X,Y) \to Y$  defined by $e_x(f)=f(x)$ is continuous with respect to $\tau_p$ topology  in $LB(X,Y)$ \cite{Ke}, and thus it is continuous also with respect to $\tau_{UC}$ topology in $LB(X,Y)$ and so the image $\mathcal{E}[x]$ of $\mathcal{E}$ is compact and thus is bounded.
	
	For the proof of finite equicontinuity of $\mathcal{E}$ we use an idea of the proof of Theorem 5.7 in \cite{HM}. Let $x\in X$ and let $U$ be an open neighbourhood of $x$ such that $\overline U=A$ is compact. Let $\epsilon >0$, we define a finite family $\mathcal B$ of subsets of $X$ as follows.
	By the compactness of $\mathcal{E}$ in  $(LB(X,Y),\tau_{UC})$, there are functions $f_1,\ ...,\ f_n\in\mathcal{E}$ such that
	$$\mathcal{E}\subset W(f_1,A,\frac\epsilon 3)\cup ... \cup W(f_n,A,\frac\epsilon 3).$$
	Since every function from $\mathcal{E}$ is locally bounded and $A$ is compact, for every $i\in\{1,\ 2,\ ...,\ n\}$ the set $f_i(A)$ is bounded. Fix point $y_0$ $Y$ and let $r>0$ be such that the set $f_1(A)\cup ... \cup f_n(A)$ is contained in the closed ball $B(y_0,r)$.
	Let $V_1,\ V_2,\ ...,\ V_m$ be a finite open cover of $B(y_0,r)$, where radius of members of this cover is less than $\frac\epsilon 3$. For every $i\in\{1,\ 2,\ ...,\ n\}$, $j\in\{1,\ 2,\ ...,\ m\}$ put $B_j^i=f^{-1}_i(V_j)$.
	
	Let $\mathcal{F}$ be the finite set of all functions from $\{1,\ 2,\ ...,\ n\}$ to $\{1,\ 2,\ ...,\ m\}$.
	For every $g \in \mathcal{F}$ put $B_g=U\cap B_{g(1)}^1\cap ...\cap B_{g(n)}^n.$
	
	Now put
	$$\mathcal{B}=\{B_g:\ g\in\mathcal{F}\}.$$
	
	We show that $\cup\mathcal B$ is a neighbourhood of $x$. Let $z\in U$, then  there is $g\in\mathcal{F}$ such that $f_i(z)\in V_{g(i)}$ and thus $z\in B_{g(i)}^i$ for every $i\in\{1,\ 2,\ ...,\ n\}$.
	
	Finally, let $f\in\mathcal{E}$, let $B\in\mathcal B$ and let $p,q\in B$. There is a $g\in\mathcal{F}$ such that $B\subset B_{g(1)}^1\cap ...\cap B_{g(n)}^n$.
	Since $\mathcal{E}$ is subset of $W(f_1,A,\frac\epsilon 3)\cup ... \cup W(f_n,A,\frac\epsilon 3)$ there exists  $i\in\{1,\ 2,\ ...,\ n\}$ such that $f\in W(f_i,A,\frac\epsilon 3)$. Thus $d(f(p),f_i(p))<\frac\epsilon 3$ and $d(f(q),f_i(q))<\frac\epsilon 3$. Because $p,q\in B_{g(i)}^i$,   we have that $f_i(p)\in V_{g(i)}$ and $f_i(q)\in V_{g(i)}$.  Then 		
	\begin{align*}
	&d(f(p),f(q))\leq\\ \leq
	&d(f(p),f_i(p))+d(f_i(p)),f_i(q))+
	d(f_i(q)),f(q))<\\ <&\frac\varepsilon 3+\frac\varepsilon 3+\frac\varepsilon 3 =\varepsilon .
	\end{align*}
	Thus $\mathcal{E}$ is finitely equicontinuous.
\end{proof}

\bigskip\bigskip

\begin{thm}\label{AAA9}
	Let $X$ be a locally compact space, $(Y,d)$ be a  boundedly compact metric space and $\mathcal F$ be a closed subset of $(LB(X,Y),\tau_{UC})$.  A subset $\mathcal{E} \subset (\mathcal F,\tau_{UC})$  is compact if and only if it is closed, pointwise bounded and finitely equicontinuous.
\end{thm}

\bigskip

We will now apply Theorem \ref{AAA9} to some special subfamilies of the space of locally bounded functions.

\bigskip

Let $X, Y$ be topological spaces. A function $f: X \to Y$ is quasicontinuous \cite{Ne} at $x \in X$ if for every open set $V \subset Y, f(x) \in V$ and open set $U\subset X$, $x \in U$ there is a nonempty open set $W \subset U$ such that $f(W) \subset V$. If $f$ is quasicontinuous at every point of $X$, we say that $f$ is quasicontinuous.

We say that a subset of $X$ is quasi-open (or semi-open) \cite{Ne} if it is contained in the closure of its interior. Then a function $f: X \to Y$ is quasicontinuous if and only if $f^{-1}(V)$ is quasi-open for every open set $V\subset Y$.

\bigskip

Let $X$  be a topological space and $(Y,d)$ be a metric space. A function $f: X \to Y$ is said to be cliquish (apparentee) \cite{Th} if for any $\epsilon > 0$  and any nonempty open set $U \subset X$ there is a nonempty open set $O \subset U$ such that $d(f(x), f(y)) < \epsilon$ for all
$x, y  \in O$.

It is well known (and
easily seen) that quasicontinuous functions are cliquish, and cliquish functions are continuous at every point of a residual subset of $X$ (and vice versa if $X$ is a Baire space).

The notion of quasicontinuity of real-valued functions of real variable was introduced by Kempisty in \cite{Kem}. The property of quasicontinuity was perhaps the first time used by Baire in \cite{Ba} in the study of points of continuity of separately continuous functions. There is now a rich literature concerning quasicontinuity, see for example \cite{Ne}, \cite{KKM} and references therein.

Quasicontinuous functions are very important in many areas of mathematics. They found applications in the study of minimal usco and minimal cusco maps \cite{HH1, HH3}, in the study of topological groups \cite{Bou, Moors1}, in proofs of some generalizations of Michael's selection theorem \cite{GB},  in the study of extensions of densely defined continuous functions \cite{Ho}, in the study of dynamical systems \cite{CFM}.

\bigskip

It is a well-known fact that a uniform limit of a net of quasicontinuous
functions from a topological space $X$ into a metric space $(Y,d)$ is quasicontinuous \cite{Ne}. It is easy to verify that  a uniform limit of a net of cliquish
functions from a topological space $X$ into a metric space $(Y,d)$ is cliquish.

\bigskip
\bigskip

Baire functions from a topological space $X$ into a metric space $Y$ \cite{Hu, Ha} are  called the analytically
representable functions in \cite{Ku}. By Baire functions  we mean the smallest family of functions
from $X$ to $Y$ containing the continuous functions which is  closed with respect to
the pointwise limits of sequences of functions belonging to it. We divide this family
into Baire classes $B_{\alpha}(X,Y)$, for $0 \le \alpha  < \omega_1$ ($\omega_1$ = the first uncountable
ordinal), by defining
$B_0(X,Y$) = the family of all continuous functions from $X$ to $Y$, and

$B_{\alpha}(X,Y$)  = the family of all limits of pointwise convergent sequences of functions from $\bigcup_{\beta < \alpha} B_{\alpha}(X,Y)$.

\bigskip
\bigskip

\begin{cor}
	Let $X$ be a locally compact space, $(Y,d)$ be a boundedly compact metric space and $Q^*(X,Y)$ be the set of all quasicontinuous locally bounded functions from $X$ to $Y$.  A subset $\mathcal{E} \subset (Q^*(X,Y),\tau_{UC})$  is compact if and only if it is closed, pointwise bounded and finitely equicontinuous.
\end{cor}
\begin{proof}
	By Proposition 2.1 in \cite{HH4} $Q^*(X,Y)$ is a closed subset of $(LB(X,Y),\tau_{UC})$.
\end{proof}

\begin{cor}
	Let $X$ be a locally compact space, $(Y,d)$ be a boundedly compact metric space and $A^*(X,Y)$ be the set of all cliquish locally bounded functions from $X$ to $Y$.  A subset $\mathcal{E} \subset (A^*(X,Y),\tau_{UC})$  is compact if and only if it is closed, pointwise bounded and finitely equicontinuous.
\end{cor}
\begin{proof}
	It is easy to verify that $A^*(X,Y)$ is a closed subset of $(LB(X,Y),\tau_{UC})$.
\end{proof}

\begin{cor}
	Let $X$ be a locally compact space and $U^*(X)$ be the set of all  locally bounded upper semicontinuous functions from $X$ to $\Bbb R$.  A subset $\mathcal{E} \subset (U^*(X),\tau_{UC})$  is compact if and only if it is closed, pointwise bounded and finitely equicontinuous.
\end{cor}
\begin{proof}
	It is easy to verify that $U^*(X)$ is a closed subset of $(LB(X,\Bbb R),\tau_{UC})$.
\end{proof}

\begin{cor}
	Let $X$ be a locally compact space and $L^*(X)$ be the set of all  locally bounded lower semicontinuous functions from $X$ to $\Bbb R$.  A subset $\mathcal{E} \subset (L^*(X),\tau_{UC})$  is compact if and only if it is closed, pointwise bounded and finitely equicontinuous.
\end{cor}
\begin{proof}
	It is easy to verify that $L^*(X)$ is a closed subset of $(LB(X,\Bbb R),\tau_{UC})$.
\end{proof}

\begin{cor}
	Let $X$ be a  compact space and $B^*_{\alpha}(X,\Bbb R)$ be the set of all bounded
	functions of Baire class $\alpha$  from $X$ to $\Bbb R$.  A subset $\mathcal{E} \subset (B^*_{\alpha}(X,\Bbb R),\tau_U)$  is compact if and only if it is closed, pointwise bounded and finitely equicontinuous.
\end{cor}
\begin{proof}
	By \cite{SSN}  2.10.3  if  $(f_n)_n$ is a sequence in $ B_\alpha(X,\Bbb R)$ which uniformly converges to $f$, then $f \in B_\alpha(X,\Bbb R)$.
\end{proof}

\begin{cor}
	Let $X$ be a  compact metrizable  space, $(Y,d)$ be a boundedly compact metric space  and $B^*_{\alpha}(X,Y)$ be the set of all bounded
	functions of Baire class $\alpha$  from $X$ to $Y$.  A subset $\mathcal{E} \subset (B^*_{\alpha}(X,Y),\tau_U)$  is compact if and only if it is closed, pointwise bounded and finitely equicontinuous.
\end{cor}
\begin{proof}
	By \cite{Kech} 24.4 i) if $X,Y$ are separable metrizable spaces, $d$ is a compatible metric for $Y$ and $(f_n)_n$ is a sequence in $ B_\alpha(X,Y)$ which uniformly converges to $f$ with respect to $d$, then $f \in B_\alpha(X,Y)$.
\end{proof}

\section{Necessary conditions of compactness of some families $\mathcal E$}

In this part we will describe the family $\mathcal B$ in the definition of finite equicontinuity for such families  $\mathcal E$ in $LB(X,Y)$ which consist of functions with a dense set of continuity points.

\bigskip
We will use the following definition introduced in  \cite{Hol} for the space of locally bounded functions.

\begin{defn}
	Let $X$ be a topological space and $(Y,d)$ be a metric space. We say that a subset $\mathcal{E}$ of $F(X,Y)$ is densely equiquasicontinuous at a point $x$ of $X$ provided that for every $\epsilon >0$, there exists a finite family $\mathcal B$ of nonempty subsets of $X$ which are either quasi-open or nowhere dense such that $\cup\mathcal B$ is a neighbourhood of $x$ and such that for every $f\in\mathcal{E}$, for every $B\in\mathcal B$ and for every $p,q\in B$, $d(f(p),f(q))<\epsilon$. Then $\mathcal{E}$ is densely equiquasicontinuous provided that it is densely equiquasicontinuous at every point of $X$.
\end{defn}

\bigskip

\begin{rem}
	In Definition 3.1, replacing quasi-open by open gives an equivalent concept.
\end{rem}

\bigskip

\begin{thm}
	Let $X$ be a locally compact space, $(Y,d)$ be a  boundedly compact metric space and $\mathcal E$ be a compact subfamily of $(LB(X,Y),\tau_{UC})$ which consists of functions with a dense set of continuity points. Then $\mathcal E$ is densely equiquasicontinuous.
\end{thm}

\begin{proof}
	We will use an idea from Theorem 3.8 in \cite{HH4}, for a reader's convenience we will write the proof. We show that $\mathcal E$ is densely equiquasicontinuous. Let $x \in X$ and let $\epsilon > 0$. We will define a finite family $\mathcal B$ of  subsets of $X$  which are either open or nowhere dense as follows. There is an open set $U(x)$ containing $x$ such that $A =\overline{U(x)}$ is compact. There are functions $f_1, f_2, ... f_n \in \mathcal E$ such that
	$$\mathcal E \subseteq \cup \{W(A,\epsilon/3)[f_j]:\ j\in\{1,..., n\}\}.$$
	
	Since $f_j$ is locally bounded, for every $j  \in \{1,..., n\}$, $\overline{f_j(A)}$ is compact. Let $\mathcal V = \{V_1, ..., V_m\}$ be a finite open cover of $\cup\{ \overline{f_j(A)}:\ j\in\{1,..., n\} \}$ where radius of members of this cover is less than $\epsilon/3$.
	
	For each $i\in\{1,..., m\}$ and $j\in\{1,..., n\}$, let
	$$H_i^j  = \{z \in U(x) \cap C(f_j): f_j(z) \in V_i\},$$
	where $C(f_j) = \{z \in X:\ f_j\ \text{is continuous at}\ z\}$. By the assumption $C(f_j)$ is dense in $X$. For each $z \in H_i^j$, let $W_i^j(z)$ be an open neighbourhood of $z$ such that $f_j(W_i^j(z)) \subseteq V_i$.  Denote by
	$$W_i^j = \bigcup \{W_i^j(z):\ z \in H_i^j\}.$$
	
	Put
	$$\mathcal U = \{W_{i_1}^1 \cap ... \cap W_{i_n}^n: i_j \in \{1,..., m\}\ \text{for each}\ j \in \{1,..., n\}\}.$$
	
	Let $U \in \mathcal U$ and let
	$$P_i^j(U) = \{z \in \overline U \setminus U:\ f_j(z) \in V_i\}.$$
	
	Now, put
	$$\mathcal P = \{P_{i_1}^1(U) \cap ... \cap P_{i_n}^n(U): i_j \in \{1,..., m\}\ \text{for each}\ j \in \{1,..., n\}, U \in \mathcal U\}.$$
	
	We claim that $U(x)\subset (\cup\mathcal U)\cup (\cup\mathcal P)$. Let $z\in U(x)$, if for every $j \in \{1,..., n\}$ there is  $i_j \in \{1,..., m\}$ such that $z\in W_{i_j}^j$, then $z\in W_{i_1}^1 \cap ... \cap W_{i_n}^n$ and so $z\in\cup\mathcal U$. Let $z\in U(x) \setminus \cup\mathcal U$. 	
	Let $G$ be a  neighbourhood of $z$. Choose some $u\in G\cap C(f_1)\cap ...\cap C(f_n)$. Then for every $j\in \{1,..., n\}$, $u\in W_{i_j}^j$ for some $i_j\in\{1,..., m\}$. Thus  $u\in G\cap W_{i_1}^1 \cap ... \cap W_{i_n}^n$ and so $z\in\overline{\cup\mathcal U}$. Hence there is $U\in\mathcal U$ such that $z\in\overline U\setminus U$.  From this follows that for every $j \in \{1,..., n\}$ there is  $i_j \in \{1,..., m\}$ such that $z\in P_{i_j}^j(U)$ and then $z\in\cup \mathcal P$. Hence $(\cup\mathcal U)\cup (\cup\mathcal P)$ is a neighbourhood of $x$.
	
	Finally, let $\mathcal{B}$ be the family containing all nonempty sets from $\mathcal U$ and $\mathcal P$.
	One can check that for every $f \in \mathcal E$, for every $B\in\mathcal{B}$, and for every $p, q \in B$, $d(f(p),f(q))<\epsilon$.
\end{proof}

\bigskip

\begin{cor} (\cite{Hol, HH4})
	Let $X$ be a locally compact space, $(Y,d)$ be a  boundedly compact metric space and $\mathcal E$ be a compact subfamily of $(Q^*(X,Y),\tau_{UC})$. Then $\mathcal E$ is densely equiquasicontinuous.
\end{cor}
\begin{proof}
	By \cite{SSN}  2.11.3  if $f:X \to Y$ is quasicontinuous, then the set $C(f)$ of the continuity points of $f$ is a dense $G_\delta$-set in $X$.
\end{proof}

\begin{cor}
	Let $X$ be a locally compact space, $(Y,d)$ be a  boundedly compact metric space and $\mathcal E$ be a compact subfamily of $(A^*(X,Y),\tau_{UC})$. Then $\mathcal E$ is densely equiquasicontinuous.
\end{cor}
\begin{proof}
	By \cite{SSN}  2.11.3  if $f:X \to Y$ is cliquish, then the set $C(f)$ of the continuity points of $f$ is a dense $G_\delta$-set in $X$.
\end{proof}

\begin{cor}
	Let $X$ be a locally compact space and $\mathcal E$ be a compact subfamily of $(U^*(X),\tau_{UC})$. Then $\mathcal E$ is densely equiquasicontinuous.
\end{cor}
\begin{proof}
	By \cite{En}  if $f \in  U^*(X)$, then the set $C(f)$ of the continuity points of $f$ is a dense $G_\delta$-set in $X$.
\end{proof}

\begin{cor}
	Let $X$ be a locally compact space and $\mathcal E$ be a compact subfamily of $(L^*(X),\tau_{UC})$. Then $\mathcal E$ is densely equiquasicontinuous.
\end{cor}
\begin{proof}
	By \cite{En} if $f \in  L^*(X)$, then the set $C(f)$ of the continuity points of $f$ is a dense $G_\delta$-set in $X$.
\end{proof}

\begin{cor}
	Let $X$ be a  compact space and $\mathcal E$ be a compact subfamily of $(B^*_1(X,\Bbb R),\tau_U)$. Then $\mathcal E$ is densely equiquasicontinuous.
\end{cor}
\begin{proof}
	By \cite{SSN}  2.10.1  if $f \in  B^*_1(X,\Bbb R)$, then the set $C(f)$ of the continuity points of $f$ is a dense $G_\delta$-set in $X$.
\end{proof}

\bigskip
\bigskip

We say that a subset $\mathcal E$ of $F(X,Y)$ is uniformly bounded provided
$\bigcup\{f(X):\ f \in \mathcal E\}$ is contained in a compact set in $(Y,d)$. Notice that if $Y = \Bbb R$ it coincides with the classical notion of uniform boundedness.
\bigskip

If $(Y,d)$ is a metric space, denote by $K(Y)$ the space of all nonempty compact sets in $(Y,d)$.
Let $H_d$ be the Hausdorff metric induced by $d$ on $K(Y)$. It is known that

\bigskip
\centerline{$H_d(A,B) = \inf\{\epsilon > 0:\ A \subset S(B,\epsilon)$ and $B \subset S(A,\epsilon)\}.$}

\bigskip

\begin{prop}
	Let $X$ be a compact topological space and $(Y,d)$ be a boundedly compact metric space.  If $\mathcal E$ is compact in $(LB(X,Y),\tau_U)$, then $\mathcal E$ is uniformly bounded.
\end{prop}
\begin{proof}
	The mapping
	$$M: (LB(X,Y),\tau_U) \to (K(Y),H_d)$$
	defined by $M(f) = \overline{f(X)}$ is continuous.
	
	Let $\{f_n:\ n \in\Bbb N\}$ converge to $f$ in $(LB(X,Y),\tau_U)$. We will show that $\{\overline{f_n(X)}:\ n \in\Bbb N\}$ converges to $\overline{f(X)}$ in $(K(Y),H_d)$. Let $\epsilon > 0$. There is $n_0 \in\Bbb N$ such that
	
	\bigskip
	\centerline{$d(f_n(x),f(x))< \epsilon/2$ for every $x \in X$ and for every $n \ge n_0$.}
	
	\bigskip
	We claim that $H_d(\overline{f_n(X)},\overline{f(X)}) < \epsilon$ for every $n\ge n_0$. Let $n \ge n_0$ and let $y \in \overline{f_n(X)}$. There is $x \in X$ such that $d(y,f_n(x)) < \epsilon/2$. Then $d(y,f(x)) < \epsilon$, i.e. $\overline{f_n(X)} \subset S(f(X),\epsilon)$. The opposite inclusion is similar.
	
	Thus $M(\mathcal E)$ is compact in $(K(Y),H_d)$. Then $\bigcup_{f \in \mathcal E} \overline{f(X)}$ is compact in $Y$ \cite{Mi}. (The Vietoris topology and the Hausdorff metric topology coincide on $K(Y)$.)
\end{proof}

\bigskip
\bigskip

If $X$ is a compact space and $(Y,d)$ is a metric space, then $\tau_{UC} = \tau_U$ on $F(X,Y)$. Since $(LB(X,Y),\tau_U)$ is metrizable, we have the following interesting consequences of our results:

\bigskip

\begin{cor}
	Let $X$ be a compact space and $(Y,d)$ be a boundedly compact metric space. If $\{f_n:\ n \in\Bbb N\}$ is a sequence of uniformly bounded finitely equicontinuous quasicontinuous functions from $X$ to $Y$, then there is a uniformly convergent subsequence $\{f_{n_k}:\ k \in\Bbb N\}$.
\end{cor}

\begin{cor}
	Let $X$ be a compact space and $(Y,d)$ be a boundedly compact metric space. If $\{f_n:\ n \in\Bbb N\}$ is a sequence of uniformly bounded finitely equicontinuous cliquish functions from $X$ to $Y$, then there is a uniformly convergent subsequence $\{f_{n_k}:\ k \in\Bbb N\}$.
\end{cor}

\begin{cor}
	Let $X$ be a compact space. If $\{f_n:\ n \in\Bbb N\}$ is a sequence of uniformly bounded finitely equicontinuous upper semicontinuous functions from $X$ to $\Bbb R$, then there is a uniformly convergent subsequence $\{f_{n_k}:\ k \in\Bbb N\}$.
\end{cor}

\begin{cor}
	Let $X$ be a compact space. If $\{f_n:\ n \in\Bbb N\}$ is a sequence of uniformly bounded finitely equicontinuous lower semicontinuous functions from $X$ to $\Bbb R$, then there is a uniformly convergent subsequence $\{f_{n_k}:\ k \in\Bbb N\}$.
\end{cor}

\begin{cor}
	Let $X$ be a compact space. If $\{f_n:\ n \in\Bbb N\}$ is a sequence of uniformly bounded finitely equicontinuous functions of Baire class $\alpha$ from $X$ to $\Bbb R$, then there is a uniformly convergent subsequence $\{f_{n_k}:\ k \in\Bbb N\}$.
\end{cor}

\bigskip
\bigskip

% ------------------------------------------------------------------------

\subsection*{Acknowledgment}
Authors would like to thank to grant Vega 2/0006/16

% ------------------------------------------------------------------------

\begin{thebibliography}{1}
\bibitem{Ar1} C. Arzel\` a,  \emph{Sulle funzioni di linee},  Mem. Accad. Sci. Ist. Bologna Cl. Sci. Fis. Mat., \textbf{5} (5) (1895) 55--74.

\bibitem{Ar2} C. Arzel\` a, \emph{Un'osservazione intorno alle serie di funzioni},  Rend. Dell' Accad. R. Delle Sci. Dell'Istituto di Bologna (1882-1883), 142--159.


\bibitem{As} G. Ascoli, \emph{Le curve limiti di una varietà data di curve},  Atti della R. Accad. Dei Lincei Memorie della Cl. Sci. Fis. Mat. Nat., \textbf{18} (3)  (1883-1884) 521--586.

\bibitem{Ba} R. Baire, \emph{Sur les functions des variables
	reelles}, Ann. Mat. Pura Appl. \textbf{3} (1899), 1--122.

\bibitem{Be} G. Beer, \emph{Topologies on closed and closed convex
	sets}, Kluwer Academic Publishers, 1993.

\bibitem{Bou} A. Bouziad, \emph{Every \v Cech-analytic Baire semitopological group is a topological group}, Proc. Amer. Math. Soc. \textbf{124} (1996), 953--959.

\bibitem{CFM} A. Crannell, M. Frantz, M. LeMasurier, \textit{Closed relations and equivalence classes of quasicontinuous functions}, Real Analysis Exchange \textbf{31} (2006/2007), 409--424.


\bibitem{En} R. Engelking, \textit{General Topology}, Helderman Verlag Berlin 1989.


\bibitem{GB} J.R. Giles, M.O. Bartlett, \textit{Modified continuity and a generalization of Michael's selection theorem}, Set-Valued Analysis \textbf{1} (1993), 365--378.


\bibitem{Ha} R.W. Hansell, \emph{On Borel mappings and Baire functions}, Trans. Amer. Math. Soc. \textbf{174} (1994), 195--211.

\bibitem{Hu} F. Hausdorff, Mengenlehre, 3rd ed., de Gruyter, Berlin, 1937; English transl., Chelsea, New York,
1957.M R 19, 111.

\bibitem{HH1} \v L. Hol\'a, D. Hol\'y, \emph{Minimal usco maps,
	densely continuous forms and upper semicontinuous functions}, Rocky Mount. Math. J. \textbf{39} (2009), 545--562.


\bibitem{HH3} \v L. Hol\'a, D. Hol\'y, \emph{New characterization of minimal CUSCO
	maps},  Rocky Mount. Math. J. \textbf{44} (2014), 1851--1866.

\bibitem{HH4} \v L. Hol\'a, D. Hol\'y, \textit{Quasicontinuous subcontinuous functions and compactness}, Mediterranean J. Math. \textbf{13} (2016), 4509--4518.

\bibitem{HH5} \v L. Hol\'a, D. Hol\'y, \emph{Minimal usco and minimal cusco maps and compactness},  J. Math. Anal. Appl. \textbf{439} (2016), 737--744.

\bibitem{HH6} \v L. Hol\'a, D. Hol\'y, \textit{Quasicontinuous functions and compactness}, Mediterranean J. Math. \textbf{14} (2017),


\bibitem{Ho} \v L. Hol\'a, Functional characterization of $p$-spaces, Central European Journal of Math. \textbf{11} (2013), 2197-2202.

\bibitem{Hol} D. Hol\' y, \emph{Ascoli-type theorems for locally bounded quasicontinuous functions, minimal usco and minimal cusco maps}, Ann. Funct. Anal. \textbf{6} (2015), no. 3, 29--41, http://projecteuclid.org/afa

\bibitem{HM} S.T. Hammer, R.A. McCoy,  \emph{Spaces of densely
	continuous forms}, Set-Valued Analysis \textbf{5} (1997), 247--266.

\bibitem{Kech} A.S. Kechris, \emph{Classical Descriptive Set Theory}, Springer-Verlag, 1987.

\bibitem{Ke} J.L. Kelley, \emph{General Topology}, Van Nostrand, New York, 1955.

\bibitem{Kem} S. Kempisty, \emph{Sur les fonctions quasi-continues}, Fund. Math. \textbf{19} (1932), 184--197.

\bibitem{KKM} P.S. Kenderov, I.S. Kortezov, W.B. Moors, \emph{Continuity
	points of quasi-continuous mappings}, Topology Appl. \textbf{109} (2001), 321--346.


\bibitem{Ku}  K. Kuratowski, Topologie. Vol. 1,   4th ed., PWN, Warsaw, 1958; English transl., Academic Press,
New York; PWN, Warsaw, 1966. MR 19, 873.

\bibitem{Mi} E. Michael, \emph{Topologies on spaces of subsets}, Trans. Amer. Math. Soc. \textbf{71} (1951), 152--182.


\bibitem{Mc} R. A. McCoy, \emph{Spaces of semicontinuous forms}, Topology Proc. \textbf{23} (1998) 249--275.

\bibitem{Moors1} W.B. Moors, \textit{Any semitopological group that is homeomorphic to a product of \v Cech-complete spaces is a topological group}, Set-Valued Var. Anal. \textbf{21} (4) (2013), 627-633.


\bibitem{Ne} T. Neubrunn, \emph{Quasi-continuity}, Real
Anal. Exchange \textbf{14} (1988), 259--306.

\bibitem{SSN} M. \v Svec, T. \v Sal\'at,  T. Neubrunn, \emph{Matematick\'a anal\'yza funkci\'{\i} re\'alnej premennej}, Alfa, Bratislava, 1986.


\bibitem{Th}  H.P. Thielman, \emph{Types of functions}, Amer. Math. Monthly \textbf{60} (1953) 156--161.

\bibitem{Va} H. Vaughan, \emph{On locally compact metrizable spaces},  Bull. Amer. Math. Soc. \textbf{43} (1937), 532--535.
\end{thebibliography}
\end{document}